\documentclass[11pt,a4paper,leqno]{amsart}
\usepackage[applemac]{inputenc}
\usepackage[T1]{fontenc}
\usepackage[english]{babel}
\usepackage{amsmath,amssymb,amsfonts,amsthm,amsopn,mathrsfs}
\usepackage{calc}
\usepackage{cite}
\usepackage{color}
\usepackage{epsfig,verbatim}
\usepackage{subfigure}
\usepackage{mathtools}
\usepackage{latexsym}
\usepackage{pdfpages}
\usepackage{graphicx}
\usepackage{enumerate}
\usepackage{tikz,xifthen}
\usetikzlibrary{matrix}
\usepackage{hyperref}
\hypersetup{pdfstartview={XYZ null null 1.00}}
\newcommand{\Z}{\mathbb{Z}}

\newcommand{\ud}{\mathrm{d}}

\newcommand{\la}{\langle}
\newcommand{\mc}{\lesssim}

\newcommand{\ra}{\rangle}

\newcommand{\N}{\mathbb{N}}
\newcommand{\f}{\hat{f}}

\newcommand{\ld}{L^{2}(\R^{d})}

\newcommand{\R}{\mathbb{R}}
\newcommand{\s}{\mathcal{S}(\R^{d})}
\newcommand{\st}{\mathcal{S}'(\R^{d})}
\newcommand{\F}{\mathcal{F}}
\newcommand{\C}{\mathbb{C}}

\newcommand{\om}{\omega}
\newcommand{\al}{\alpha}
\newcommand{\be}{\beta}
\newcommand{\p}{\partial}
\newcommand{\me}{\mathrm{e}}

\newtheorem{theorem}{Theorem}[section]
\newtheorem{prop}[theorem]{Proposition}

\newtheorem{Cor}[theorem]{Corollary}
\newtheorem{lemma}[theorem]{Lemma}

\newtheorem{remark}[theorem]{Remark}

\title[Gabor Decomposition of Evolution Operators and Applications]{Gabor Frame Decomposition of Evolution Operators and Applications}
\author[M.Berra]{Berra Michele}
\address{Dipartimento di Matematica ``Giuseppe Peano'', Universit\`a degli Studi di Torino, Via Carlo Alberto 10, 10123 Torino (TO), Italy.}\email{michele.berra@unito.it}
\begin{document}
\selectlanguage{english}
\begin{abstract}
We compute the Gabor matrix for Schr\"odinger-type evolution operators.
Precisely, we analyze the Heat Equation, already presented in  \cite{2012arXiv1209.0945C}, giving the exact expression of the Gabor matrix which leads to better numerical evaluations.
Then, using asymptotic integration techniques, we obtain  an upper  bound for the Gabor matrix in one-dimension for the generalized Heat Equation,  new in the literature. Using Maple software, we show numeric representations of the coefficients' decay.
Finally, we show the super-exponential decay of the coefficients of the  Gabor matrix for the Harmonic Repulsor, together with some numerical evaluations.
This work is the natural prosecution of the ideas presented in \cite{2012arXiv1209.0945C} and \cite{MR2502369}. \vspace{3em} 
\end{abstract}
\maketitle

\section{Introduction}
The Gabor frame theory has been developed in the last fifty years in order to give a discrete time-frequency representation of a signal in the phase space.  We shall use Gabor frames to give a discrete time frequency representation of operators. This new field of research, started in \cite{2012arXiv1209.0945C, MR2502369} is the main topic of our study. For simplicity, we limit the study to Gabor frames defined on a regular lattice $\Lambda= \al\Z^{d}\times\be\Z^{d}, \al,\be >0$, but more general lattices may be considered.\\
\indent Precisely, given a window function
$\; g\in\s\backslash\{0\} $, i.e. the Schwartz class, and a lattice $\Lambda \coloneqq \al\Z^{d}\times\be\Z^{d}$ with $\al,\be>0$ , a Gabor system is defined as 
\[\mathcal{G}(g,\al,\be)\,\coloneqq\,\big\{g_{m,n} =M_{n}T_{m}g,\;  (m,n)\,\in \Lambda\big\},
\] 
where $M_{n}g(x) = \me^{2\pi i n x}g(x)$ and $T_m g(x) = g(x-m)$.
 $\mathcal{G}(g,\al,\be)$ is a Gabor frame for  $\ld$ if and only if there exist $0< A\leq B <+\infty$:\[
A\|f\|_{2}^{2}\leq \sum_{(m,n)\,\in\, \Lambda}|\la f,g_{m,n}\ra|^{2} \leq B\|f\|_{2}^{2}, \qquad \forall f \in\ld.\]
The previous inequality implies the representations 
\begin{equation}\label{gfrel1}  f = \sum_{m,n}\la f,g_{m,n}\ra \gamma_{m,n}  \qquad \mbox{or}\qquad \sum_{m,n}\la f,\gamma_{m,n}\ra g_{m,n}, \qquad \forall f\,\in\, \ld,\end{equation} where $\{\gamma_{m,n}\}$ is a Gabor frame called dual frame and \eqref{gfrel1} yields with unconditional convergence in $\ld$.\\
\indent  Gabor frames turned out to be the appropriate tool for many problems in time-frequency analysis, especially  signal processing and imaging problems with related numerical issues, see for example \cite{christensen2011gabor} \cite{
grafakos2008gabor},\cite{ onchis2010gabor}, \cite{Strohmer2006237} and the references therein.  From the theoretical point of view, Gabor frames are used to investigate  Fourier Integral Operators, see  \cite{2012arXiv1209.0945C}, \cite{MR3007846} and \cite{MR2556742}, and in particular Pseudo Differential  Operators\cite{MR2498351}, by representing such operators  as infinite matrices and studying the decay of their entries.\\
\indent Following the ideas presented in  \cite{2012arXiv1209.0945C}, we perform the Gabor decomposition of evolution operators by computing the Gabor matrix and giving numerical evidences of the coefficients' decay.\\
\indent Given a bounded operator acting on $\ld$, say $T:\ld\rightarrow \ld$, the idea is to discretize the action of $T$ on a given function $f\in \ld$. Roughly speaking, we want to give a clear meaning to the following diagram:
%\[\begin{diagram}
%T: \ld &\rTo &\ld\\
%\dTo_{} & &\dTo_{}\\
%l^2 &\rTo^{} & l^2
%\end{diagram}
%\]
\newline
\begin{center}
\begin{tikzpicture}
\matrix (m) [matrix of math nodes,row sep=3em,column sep=4em,minimum width=2em]
  {
     \ld & \ld \\
     l^{2} & l^{2} \\};
  \path[-stealth]
    (m-1-1) edge node [left] {Analysis} (m-2-1)
            edge [right] node [above] {$T$} (m-1-2)
    (m-2-1.east|-m-2-2) edge node [below] {}
            node [above] {$T_{m,n,m',n'}$} (m-2-2)
    (m-2-2) edge node [right] {Synthesis} (m-1-2);
\end{tikzpicture}
\end{center}

\noindent where $T_{m,n,m',n'}$ is a matrix operator that we are about to discuss.\\ \indent  The first idea is to represent the signal $f$ via Gabor frames and then apply the operator $T$ on it. Precisely,
given $ f =\sum_{m,n}\la f,\gamma_{m,n}\ra g_{m,n}$ we have
\[Tf =T\left(\sum_{m,n}\la f,\gamma_{m,n}\ra g_{m,n}\right) = \sum_{m,n}\la f,\gamma_{m,n}\ra Tg_{m,n}.\]
On the other hand, since $Tf\in \ld$, we can decompose it as  
\[Tf =\sum_{m',n'}\la Tf,g_{m',n'}\ra \gamma_{m',n'},\]
with $\{g_{m',n'}\}_{m',n'}$ and $\{\gamma_{m',n'}\}_{m',n'}$ Gabor frame and dual frame respectively.\\
Putting together these relations, one gets:
\begin{align}
Tf 	&= \sum_{m',n'}\la Tf,g_{m',n'}\ra \gamma_{m',n'}\nonumber = \sum_{m',n'}\la \sum_{m,n}\la f,\gamma_{m,n}\ra Tg_{m,n},g_{m',n'}\ra \gamma_{m',n'}\nonumber\\
	&= \sum_{m,n,m',n'}\la Tg_{m,n},g_{m',n'}\ra \la f,\gamma_{m,n}\ra \gamma_{m',n'}.\label{eqI1}
\end{align}
We have something similar to a ``kernel operator''. Indeed, using the notations of the diagram above, the kernel is  $T_{m,n,m',n'}  =\la Tg_{m,n},g_{m',n'}\ra $ and $ \la f,\gamma_{m,n}\ra \gamma_{m',n'}$ is somehow similar to the expansion of $f$ in terms of the dual frame $\gamma$.
Therefore, we can analyze the action of the operator $T$ by studying the behavior of this infinite matrix which we call the Gabor matrix of $T$.\\
\indent
The main theoretical topics on which we base our numerical evaluation are the sparsity results for the Gabor matrix of different classes of Fourier Integral Operator proved in  \cite{2012arXiv1209.0945C} and \cite{MR2556742}. According to \cite{MR2477144}, given a discrete representation of the operator $T$ of the form
\[Tf_k = \sum_{j} T_{j,k}f_j,\]
where $(T_{j,k})_{j,k}$ is a  $N\times N$ matrix,  one requires $o(N^2)$ operations to calculate $Tf_k $.
If $T$ is diagonal, the computational effort reduces to $o(N)$.\\
\indent In order to speed up the computations, we can select only the coefficients of the matrix that are relevant in the following sense: we select a small parameter $\varepsilon>0$ and consider only the coefficients $T_{j,k}$ satisfying $|T_{j,k}|>\varepsilon$. Hence, when the matrix shows off-diagonal decay, for example
\[|T_{j,k}|\leq C\left(1 + |j-k|^{2}\right)^{-\frac{N}{2}}, \] there exists $R(\varepsilon)>0$ such that $|T_{j,k}|>\varepsilon$ for $|j-k|< R(\varepsilon)$. In this case, the computational effort reduces to $o\left(N\cdot R(\varepsilon)\right)$ and if the off-diagonal decay is heavy, then $R(\varepsilon)$ can be chosen small.\\ \indent  The arguments above can be adapted to a ``well-organized'' matrix in the sense of \cite{MR2502369}. Precisely, given a canonical transformation $\chi:\R^{2d}\longrightarrow \R
^{2d}$ and an operator $T$ such that 
\[|T_{j,k}|\leq C\left(1 + |j-\chi(k)|^{2}\right)^{-\frac{N}{2}},\]  we reach fast decay far from the graph of a function $\chi$. Hence, there exists $R_{\chi}(\varepsilon)>0$ such that $|T_{j,k}|>\varepsilon$ for $|j-\chi(k)|<R_{\chi}(\varepsilon)$.\\
\indent Thus, using $\eqref{eqI1}$, the more decay ``away from the diagonal'' the more efficient is the representation of the operator T.  
Therefore, it is important to give a precise expression for the Gabor matrix and provide estimates for the coefficients' decay.\\
\indent 
We calculate the Gabor matrix for three well-known PDEs, i.e. the Heat equation, the so-called generalized Heat equation and the Harmonic Repulsor. 
The Heat equation was formerly investigated in \cite{2012arXiv1209.0945C}, we improve their estimates by giving the exact formulation of the Gabor matrix. This yield to an improvement of the coefficients'  decay rate, when compared with \cite{2012arXiv1209.0945C}.
The analysis of the generalized Heat equation is more intriguing and completely new in the literature.
We find an upper bound for the Gabor matrix exploiting all the constants that are significant from a numerical point of view.
Finally, we analyze the case of the so-called  Harmonic Repulsor giving the exact representation of the Gabor matrix. This new result yields to a super exponential decay for the coefficient that is similar to the ones of the Harmonic Oscillator, treated in \cite{MR2502369}.\\
\indent
The paper is organized as follows: we give a brief review of the main tools of time frequency analysis used in this paper, namely Gabor frames and Gelfand-Shilov classes. Using these arguments we recall the main result of \cite{2012arXiv1209.0945C} on the super exponential decay of the Gabor matrix for Pseudo Differential Operators. 
Then we introduce the definition of the metaplectic representation used to analyze the Cauchy problem for the Harmonic Repulsor.
We conclude the preliminaries by proving an interesting result of classical asymptotic integration, precisely we give an estimate for the Fourier transform of \[f = \me^{-\alpha x^{2k}}, x\in\R, \al>0,k\in \Z_{+}.\] \\ \indent Section 3 is devoted to the Heat Equation and generalizations. We start from the Cauchy problem for the Heat equation, namely
\begin{equation}
\begin{split}
&\p_{t}u -\rho\Delta_{x} u = 0, \qquad (t,x)\,\in\R\times\R^{d},\\
& u(0,x) = u_{0}(x).
\end{split}
\end{equation}
The solution can be written as
\begin{equation}u(t,x) = \sigma_{\rho}(t,D)u_{0}(x),\end{equation}
where $\sigma_{\rho}(t,D)$ is the family of Fourier multipliers 
\begin{equation}\sigma_{\rho}(t,D)f(x) = \int_{\R^{d}}\me^{2\pi i x\om} \sigma_{\rho}(t,\om)\f(\om)\ud \om,\end{equation}
with symbol 
\begin{equation}\sigma_{\rho}(t,\om) = \me^{-4\rho t \pi^{2}|\om|^{2}}, \quad \om\in\R^{d},\end{equation}
where $t\in\R$ represent the time variable and $\rho>0$ is the thermal diffusion parameter.
Taking the Gaussian as window function $g$ for the Gabor frame  we obtain a Gabor matrix satisfying
\begin{align}\begin{split}
&\left|\la \sigma_{\rho}(t,D)\pi(m,n)g , \pi(m',n')g\ra\right| = \\
&(2+4\pi\rho t )^{-\frac{d}{2}} \me^{  -\pi\left [|n|^2+|n'|^2+\frac{1}{2+4\pi\rho t}\left(|m-m'|^2+|n+n'|^2\right)\right]}.
\end{split}
\end{align}
Using this relation, we obtain faster decay of the coefficient of those obtained in \cite{2012arXiv1209.0945C} as proved in Section 5 with some numerical implementation.\\
\indent In the second part of the section, we find a bound for the decay of the coefficients of the Gabor matrix for the generalized Heat equation using the asymptotic theory developed in the preliminaries.
We obtain results that are consistent with the ones  proved in \cite{2012arXiv1209.0945C} but here we make explicit the constants involved in the calculations and give a more precise information about the decay of the coefficients.
The main result of the section reads 
\begin{equation}\vert \la \sigma_{k}(t,D)\pi(\lambda)g,\pi(\nu)g\ra\vert  \mc C_{t,k}\me^{ -\tilde{\varepsilon}_{t,k} 2^{-\frac{1}{s}}(\lambda-\nu) ^{\frac{1}{s}}}, \end{equation}
with $s = \frac{2k}{2k-1}$, $\tilde{\varepsilon}_{k,t} = \left(\frac{2k-1}{2k}\right)\left(\frac{1}{2k t}\right)^{\frac{1}{2k-1}}2^{-\frac{k}{2k-1}} $ and $C_{t,k} =  |2 k t|^{\frac{k-1}{2k-1}}$ and a suitable constant $C_{0}$ that does not depend on $t$ and $k$. Again, $\sigma_{k}(t,D)$ is the family of Fourier multiplier depending on $t$ and $k$ that solves 
\begin{equation}
\begin{split}
&\p_{t}u -\Delta^{k}_{x} u = 0, \qquad (t,x)\,\in\R\times\R,\, k\geq 1\\
& u(0,x) = u_{0}(x).
\end{split}
\end{equation}
The symbol related to the Fourier multiplier $ \sigma_{k}(t,D)$ reads
\begin{equation}
\sigma_{k}(\om) = \me^{ (-1)^{k}t (2\pi\om)^{2k}}, \quad \om\in\R.
\end{equation}
\indent Finally, the last equation to be considered is the Cauchy problem for the Harmonic Repulsor, or Harmonic Repulsive Oscillator \cite{MR578097}, namely
\begin{equation}
\begin{split}
&i\p_t u -\frac{1}{4\pi} \Delta u + \pi |x|^2= 0,\qquad (t,x) \in \R\times \R^{d},\\
&u(0,x) = u_0(x).
\end{split}
\end{equation}
The solution can be written as
\begin{equation}
u(t,x) = T_{t}\,u_{0}(x),
\end{equation}
where $T_{t}$ is the family of metaplectic operators given by
\begin{equation}
T_{t}f(x)=\frac{1}{\cosh(t)}\int_{\R^d} \me^{ \pi i \tanh(t)x^2 + \frac{2\pi i x \om}{\cosh(t)}-\pi i \tanh(t)\om^2} \f(\om)\ud \om.
\end{equation}
The treatment of the Harmonic Repulsor is similar to the Harmonic Oscillator, already treated in \cite{cordero2012time}.
We perform a direct calculus of the Gabor matrix obtaining the following super-exponential decay
\begin{align}\begin{split} &T_{m,n,m',n'} =C_t\, \me^{ -\frac{\pi}{2}\left[|m|^2+|n|^2+|n'|^2+|m'|^2 + 2\tanh(mn-m'n') -2(mm'-nn')\right]},\end{split}\end{align}
 where $C_t =  \frac{\me^{ i\psi}}{(2\cosh(t))^{\frac{d}{2}}} $ and $|\me^{ i\psi}|=1$.
\indent The last section is entirely devoted to the numerical simulation in which we give an evidence of the theoretical results by showing Maple's plot of the coefficients' decay and the spectrogram of the Gabor  matrices associated to each of the three problems. 
\subsection*{Notations}
The Schwartz class is denoted by  $\s$, the space of
tempered distributions by  $\st$.   We use the brackets $\la f,g\ra$ for the extension to $\s\times\st$ of the inner product $\la f,g\ra=\int f(t){\overline
{g(t)}}dt$ on $L^2(\R^d)$.
The scalar product on $\R^d$ is given by $xy$ for $x,y\;\in\R^d$.
The Fourier transform is normalized to be ${\hat{f}}(\omega)=\F f(\om)=\int f(t)\me^{-2\pi i t\om}dt$.
The Modulation and Translation Operators $M$ and $T$ are defined by $M_\om g(t) = \me^{2\pi i \om t}$ 
and $T_x g(t) = g(t-x)$.
Time Frequency-Shifts are denoted $M_\om T_x g(t)  = \pi(z)g$, with $z = (x,\om)\in\R^{2d}$.
The Euclidean norm of $x\in \R^d$ is given by $|x|= (x_1^2+\dots+x_d^2)^{1/2}$.
Let $\al\in\;\Z_{+}^d$ be a  multiindex. The length of $\alpha$ is $|\al| = \al_1+\dots+\al_d$. For $x\in\R^d$ the power is represented by $x^{\al}=x_1^{\al_1}\cdots x_d^{\al_d} $  
The operator of partial differentiation $\p$ is given by
\[\p^\al = \p_x^\al =\p^{\al_1}_{x_1}\cdots\p^{\al_d}_{x_d} \]
for $x\in\R^d$ and $\al\in\;\Z^d_{+}$.
The letter $C$ denotes a positive constant, not necessarily the same at every appearance.
Throughout the paper, we shall use the notation $A\lesssim B$ to
indicate $A\leq c B$ for a suitable constant $c>0$, whereas $A
\asymp B$ if $A\leq c B$ and $B\leq k A$, for suitable $c,k>0$.
We denote the  $\delta_{ij} = \begin{cases} 1, \quad i=j\\0, \quad i\neq j\end{cases}.$
We denote also $\mathcal{M}(2d,\R)$ for the set of $2d\times2d$ real-valued matrices while $GL(2d,\R)$ is the general linear group over $\mathcal{M}(2d,\R)$.
\section{Preliminaries}
\subsection{Gabor Frames}\label{GAFR}
We recall the basic concepts of the Gabor Frame theory and refer the reader to \cite{MR1843717} and \cite{MR2744776} for the details. The Gabor Frames are used to give a discrete Time-Frequency representation of signal and operators. 
Let $\Lambda := \al\Z^{d}\times\be\Z^{d}$ be a lattice on the phase space,  with $\al,\be\, >\, 0$ lattice parameters. The set of time-frequency shifts
\[\mathcal{G}(g,\al,\be) = \{\pi(\lambda)g,\, \lambda\,\in\,\Lambda\},\]
 is called a \emph{Gabor system} for $g\in\ld\backslash \{0\}$.
The Gabor system $\mathcal{G}(g,\al,\be) $ is a (Gabor) frame for $\ld$ if there exist constants $A,B> 0 $ such that 
\begin{equation}
\label{Frame_rel}
A\|f\|^2_2 \leq \sum_{\lambda\in\Lambda} |\la f,\pi(\lambda)g\ra|^2 \leq B\|f\|^2_2.
\end{equation}
If $\eqref{Frame_rel}$ holds, 
then there exists a dual window $\gamma\in\ld$, such that $\mathcal{G}(\gamma,\al,\be) $ is a frame for $\ld$ and every $f\in\ld$ can be expanded as 
\begin{equation}\label{GabExp}
f:=\sum_{\lambda\in\Lambda}\la f,\pi(\lambda)g\ra \pi(\lambda)\gamma  = \sum_{\lambda\in\Lambda}\la f,\pi(\lambda)\gamma\ra \pi(\lambda)g,
\end{equation} 
with unconditional convergence in $\ld$. 
We shall work with Gabor frames with the window $g= \me^{-\pi|x|^2},\, x\in\R^d$, i.e. a Gaussian function. Indeed, the following result characterizes Gabor frames with a Gaussian function as window.
\begin{prop}
Let $g = \me^{-\pi|x|^2}$, $x\in\R^d$. Then $\mathcal{G}(g,\al,\be)$ is a Gabor Frame for $\ld$ if and only if $\alpha\beta<1$.
\end{prop}
See \cite[Theorem 7.5.3]{MR1843717} and \cite{MR1173117,MR1173118} for the details.
\subsection{Gelfand Shilov Spaces}
The Gelfand-Shilov spaces are subspaces of the Schwartz class $\s$ introduced in \cite{MR0435832} to give  information about how fast a function $f\in\s$ and its derivatives decay at infinity. We recall only the main properties used in this paper, the main references on this topic are \cite{MR0435832} and\cite{Ro_nic}.
Given $s,r\geq 0$, $f$ is in the Gelfand-Shilow type space $S^s_r (\R^d)$ if there exist constants $A,B>0$ such that \[
|x^{\al}\p^{\be}f(x)| \mc A^{|\al|}B^{|\be|} (\al !)^r (\be !)^s, \; \al,\be,\in\Z_+^d.\]
The space $S^s_r (\R^d)$ is non-trivial if and only if $r+s > 1$ or $r+s = 1$ and $r,s >0$. 
It can be shown that the smallest non-trivial space with $r=s$ is given by $S^{1/2}_{1/2} (\R^d)$. 
We observe that $S^{s_1}_{r_1} (\R^d)\subset S^{s_2}_{r_2} (\R^d)$ for $s_1\leq s_2$ and $r_1\leq r_2$.
Moreover, we have : \[f \in S^{s}_{r} (\R^d) \Longleftrightarrow \f \in S^{r}_{s} (\R^d).\]
Therefore the spaces $S^{s}_{s} (\R^d)$ are invariant under the action of the Fourier Transform.
Functions of type $f(x)= \me^{ -a|x|^2}$, with $a>0$ belong to $S^{1/2}_{1/2} (\R^d)$. 
An equivalent condition for $f\in \s$ to be in $S^{s}_{r} (\R^d)$ is as follows:
\begin{prop}\label{GS-prop}
Let $r,s >0$ and $r+s\geq 1 $. For $f\in\s$ the following conditions are equivalent:
\begin{itemize}
\item[a)] $f \in S^{s}_{r} (\R^d)$;\\
\item[b)] There exist positive constants $h,k$ such that 
\[\|f\me^{h|x|^{1/r}}\|_{{\infty}} <\infty\quad\textrm{and}\quad \|\f \me^{k|\om|^{1/s}}\|_{{\infty}} <\infty.\]
\end{itemize}
\end{prop}
\noindent This result is contained in \cite[Theorem 6.1.6]{Ro_nic}.\\
\indent The Gelfand-Shilow spaces were used in \cite{2012arXiv1209.0945C} as symbol space to characterize  the behavior  of the Gabor matrix of the corresponding pseudodifferential operators. Precisely, the un-weighted version of \cite[Theorem 4.2]{2012arXiv1209.0945C} can be stated as follows:
\begin{theorem}
\label{G-S decay}
Let $s\geq \frac{1}{2},\, g\in S^s_s(\R^d)$ and $\sigma\in\C^{\infty}(\R^d)$.  Then the following properties are equivalent:
\begin{enumerate}
\item[a)] The symbol $\sigma$ satisfies
\begin{equation}\label{GS_symb}
|\p^{\alpha}\sigma(z)|\mc C^{|\alpha|}(\al!)^s, \quad \forall\, z\in\R^{2d}, \forall\, \al\in\Z^{2d}_+ ;
\end{equation}
\item[b)] There exists $\varepsilon>0$ such that 
\begin{equation}
\label{GS_ineq}
|\la \sigma^{W}(\pi(z)g,\pi(w)g)\ra|\mc \me^{-\varepsilon|w-z|^{\frac{1}{s}}}, \quad \forall z,w\in\R^{2d}.
\end{equation}
where $\sigma^{W}(x,D)$ indicates  the Weyl quantization defined as
\[\sigma^{W}(x,D) f(x):= \int \me^{2\pi i (x-y)\om}\sigma\left(\frac{x+y}{2},\om\right)f(y)\ud y \ud \om.\]
\end{enumerate}  
\end{theorem} 
\subsection{Metaplectic Operators}\label{section2.4}
The Metaplectic representation is a powerful tool to study certain classes of PDEs. We briefly recall the main concepts and the results needed later in this paper.
We refer to \cite{MR983366,MR1906251}.
Let $d\in\N$ being the dimension, then the Symplectic group is defined to be 
\[Sp(d,\R) = \left\{ g\in GL(2d,\R) \: \slash\: ^t\! g\mathcal{J}g = \mathcal{J}\right\}\]
where \[\mathcal{J}= \left[\begin{array}{rr}
0 & I_d \\ 
-I_d & 0
\end{array} \right].\]
Together with the group structure, it is useful to define the Symplectic algebra.
Let $d\in\N$, then the Symplectic algebra is defined to be 
\[\mathfrak{s}\mathfrak{p}(d,\R) = \left\{ \mathcal{A}\in \mathcal{M}(2d,\R) \: \slash\: \me^{t\mathcal{A}} \in  Sp(d,\R)\right\}.\]
The metaplectic representation $\mu$ is a unitary representation of (double cover of) the Symplectic group $Sp(d,\R)$ on $\ld$.
We shall study the unitary operator $\mu(t)$ giving the solution for the Harmonic Repulsor in Section  \ref{Sec4}.
Precisely \cite[Theorem 4.51]{MR983366} gives the following explicit representation for the unitary operator.
\begin{prop}
Consider $\mathcal{A} = \left(\begin{array}{rr}
A & B \\ 
C & D
\end{array} \right)\,\in\, Sp(d,\R)$ with $A\neq 0$
then
\begin{equation}\label{Metaplectic_Repr_eqt}
\mu(\mathcal{A})f(x) = (\det A)^{-1/2}\int \me^{-\pi i x\cdot CA^{-1}x + 2\pi i \om A^{-1}x + \pi i \om A^{-1}B \om}\hat{f}(\om) \ud\om
\end{equation}\end{prop}
For the general integral representation of a metaplectic operator see \cite{MR1906251}.
Consider now the following Cauchy problem 
\begin{equation}\label{PC_SC_GEN}
\begin{cases}
i\p_t u -H_{\mathcal{A}}u= 0\\
u(0,x) = u_0(x),
\end{cases}
\end{equation}
where the Hamiltonian $H_{\mathcal{A}}$ represents the Weyl quantization of a quadratic form on $\R^{2d}$.
Indeed every matrix $\mathcal{A}\in\mathfrak{s}\mathfrak{p}(d,\R)$ defines a quadratic form $\mathcal{P}(x,\om)\in\R^{2d}$ as follows:
\[\mathcal{P}(x,\om) = -\frac{1}{2} {^t\!(x,\om)} \mathcal{A}\mathcal{J}(x,\om).\]
Setting  $\mathcal{A} = \left(\begin{array}{rr}
A & B \\ 
C & ^t\! A
\end{array} \right)\,\in\, \mathfrak{sp}(d,\R)$ 
we have
\begin{equation}
\label{Eqt_form_quadr}
\mathcal{P}(x,\om) = \frac{1}{2}\om  B\om -\om Ax -\frac{1}{2}x Cx.
\end{equation}
From the Weyl quantization, the quadratic form in \eqref{Eqt_form_quadr} corresponds to the Weyl operator $\mathcal{P^{\om}_A} = \mathcal{P^{\om}_A}(D,X)$ defined by
\begin{equation}\label{Weyl_operator}
2\pi \mathcal{P^{\om}_A}= -\frac{1}{4\pi} \sum_{j,k = 1}^{d}B_{j,k}\frac{\p^2}{\p_{x_{j}x_{k}}} + i\sum_{j,k = 1}^{d}A_{j,k}x_k\frac{\p}{\p_{x_{j}}}+\frac{i}{2}\textrm{Tr}(A)-\pi\sum_{j,k = 1}^{d}C_{j,k}{x_{j}}{x_{k}.}
\end{equation}
The operator $H_{\mathcal{A}} = 2\pi \mathcal{P^{\om}_A}(D,X)$ is the Hamiltonian operator. The evolution operator which provides the solution to \eqref{PC_SC_GEN} is  then given by 
\[\me^{i t H_{\mathcal{A}}} = \mu\left(\me^{t\mathcal{A}}\right).\]
Hence, the solution to \eqref{PC_SC_GEN} is $u(t) = \me^{i t H_{\mathcal{A}}}u_0 =  \mu\left(\me^{t\mathcal{A}}\right)u_0$
\subsection{Asymptotic Integration}$\phantom{2}$\\
Throughout this section we will work in dimension $d=1$.  Consider the function $f(x) = \me^{-\al x^{2k}}, \,\al>0, k\geq 1$.\\ Let $r=\frac{1}{2k}$, it is clear that there exists $h>0$ such that $\parallel f \me^{h |x|^{\frac{1}{r}}}\parallel_\infty<+\infty$. \\\indent We shall prove that given $s = 1-\frac{1}{2k}$ there exist $\be > 0 $ such that $$\parallel \hat{f}(\om)\me^{\be|\om|^{\frac{1}{s}}}\parallel_\infty<+\infty.$$
Thus, using Proposition \ref{GS-prop}, we would obtain that $f(x)\in\mathcal{S}^{1-\frac{1}{2k}}_{\frac{1}{2k}}(\R)$ and consequently $\f\in\mathcal{S}_{1-\frac{1}{2k}}^{\frac{1}{2k}}(\R).$\\
\indent We need to represent the Fourier transform of $f$. The classic theory of Asymptotic Integration can be used to give an estimate of its behavior.
\begin{lemma}\label{Lemma1}
Let $\Omega = \R^{+}$. Given $n>0$ and $y:\R\rightarrow\C$ such that $y\in C^{n} (\R)$,  consider  $y^{(n)}$ to be the $n$-th derivative of the function $y$.   Define $r(t) = t^{\al},q(t) = \pm t^{\be}$ with $\be > -n$.
Then, given a positive constant $C$ \begin{equation}
y^{(n)}(t)-C\ q(t)\, y(t) = 0
\end{equation}
admits a fundamental system of solutions $y_{j}, \ j = 1,\ldots,n,$ such that
\begin{equation}
\label{solution}
y_{j}(t)\asymp [C(\pm t^{\be})]^{-\frac{(n-1)}{2n}} \me^{\nu_j \sqrt[n]{C} \frac{\be}{n}(\pm t)^{\frac{\beta}{n}+1}}
\end{equation}
for $t>0$ and $\nu_j$ are $n$-th roots of the unity.
\end{lemma}
The proof of this Lemma can be found in \cite{MR919406}. See also \cite{MR724721} and \cite{hinton}.
\begin{theorem}\label{Teorem_1}
Let $f(x) = \me^{-\al x^{2k}}$, with $x\in \R$,  $\al>0$ and $k\geq 1$. Then $\f$ satisfies:
\begin{equation}\label{RodNic}
\vert \f(\om)\vert \;\leq\; C_{k,\al}\me^{ -\varepsilon_{k,\al}|\om|^{\frac{2k}{2k-1}}},
\end{equation}
where $ C_{k,\al} = \frac{|2 k \al|^{\frac{k-1}{2k-1}}}{(2\pi)^{\frac{2k(k-1)}{2k-1}}}$ and $\varepsilon_{k,\al} =\delta_{k} (2\pi)^{\frac{2k}{2k-1}}\left(\frac{2k-1}{2k}\right)\left(\frac{1}{2k\al}\right)^{\frac{1}{2k-1}}$, for some positive constants $\delta_k$, as determined in the subsequent proof.
\end{theorem}
\begin{proof}
$\phantom{1}\vspace{3mm}$\newline
The idea is to use Lemma \ref{Lemma1}. 
Let us start by  differentiating the function $f(x) = \me^{-\al x^{2k}}$. This leads to $f'(x) = -2k\al x^{2k-1}f(x)$, hence $f$ satisfies the following differential equation
\begin{equation}\label{diff1}
f'(x) + 2k\al x^{2k-1}f(x) = 0.
\end{equation}
Applying the Fourier Transform to \eqref{diff1} yields to
\begin{align*}
0 &= \F(f'(x) + 2k\al x^{2k-1}f(x))(\om)\\
&= 2\pi i\om \hat{f}(\om) + 2k\al\cdot \left(\frac{i}{2\pi}\right)^{2k-1}\F\left((-2\pi i x)^{2k-1}f\right)(\om)\\
&= -\om\hat{f}(\om) + 2k\al\frac{(-1)^k}{(2\pi)^{2k}}\hat{f}^{(2k-1)},
\end{align*}
where in the latter we divided both sides of the equality by $2\pi i$.\\
Thus, the transformed equation is \begin{equation}
\label{Tr_diff1}
h^{(2k-1)}(\om) - \frac{(2\pi)^{2k}(-1)^k}{2k\al}\om h(\om)=0.
\end{equation}
Since the Fourier transform is an isomorphism on $\mathcal{S}(\R)$, there is a unique (modulo multiplicative constants) non-trivial solution $h \in \mathcal{S}(\R)$ of \eqref{Tr_diff1}, such that $h(\om) = \hat{f}(\om)$. 
Applying now  Lemma \ref{Lemma1} for $\om  >0$ with $n= 2k-1$, $q(\om)= \om$ and $C=\frac{(2\pi)^{2k}(-1)^k}{2k\al}$
the estimate of the absolute value of $h$ then becomes:
\begin{align}
|h(\om)|&{}\leq \left| \frac{(2\pi)^{2k}(-1)^k \om}{2k\al}\right|^{-\frac{k-1}{2k-1}}\me^{\max_{j}\Re\left\{\nu_j \left[\frac{(2\pi)^{2k}(-1)^k}{2k\al}\right]^{\frac{1}{2k-1}} \omega^{\frac{1}{2k-1}+1} \right\}}\label{trenta}\\
&= \frac{(2k\al)^{\frac{k-1}{2k-1}}}{(2\pi)^{\frac{2k(k-1)}{2k-1}}|\om|^{\frac{k-1}{2k-1}}}\me^{ -\delta_{k}\frac{2\pi^{\frac{2k}{2k-1}}}{({2k\al)^{\frac{1}{2k-1}}}}\left( \frac{2k-1}{2k}\right)\ \om ^{\frac{2k}{2k-1}}} \label{eq8}
\end{align} for some $\delta_{k}$. We have only to prove the inequality for $|\om|\rightarrow \infty$,  that is $ \frac{1}{|\om|}\rightarrow 0$, it follows that
\[|\hat{f}(\om)| \lesssim \frac{(2k\al)^{\frac{k-1}{2k-1}}}{(2\pi)^{\frac{2k(k-1)}{2k-1}}}\me^{ -\delta_{k}\frac{2\pi^{\frac{2k}{2k-1}}}{({2k\al)^{\frac{1}{2k-1}}}}\cdot \frac{2k-1}{2k}\om ^{\frac{2k-1}{2k}}} \qquad \forall\;\om>0,\]
where all the omitted constants are independent from $k$.\\ Since $f$ is real and even, $\f$ is real and even too. Thus, its behavior at $-\infty$ will be the one of \eqref{eq8}, as desired.
\end{proof}
\begin{remark}\label{note3}
This result prove what we claimed at the beginning of this section. Indeed, there exists $k>0$ such that $\|\hat{f}\me^{k|\om|^{\frac{1}{s}}}\|_{\infty}\leq \infty,$ for $s=1-\frac{1}{2k}$.
Using Proposition \ref{GS-prop}, we have that $f(x) = \me^{-\al(x)^{2k}}\in\mathcal{S}^s_r(\R) = \mathcal{S}^{1-\frac{1}{2k}}_{\frac{1}{2k}}(\R)$.
Consequently $\f\in\mathcal{S}^r_s(\R) = \mathcal{S}_{1-\frac{1}{2k}}^{\frac{1}{2k}}(\R).$
\end{remark}
\begin{Cor}\label{cor2}
Let $f(x) = \me^{-\be_{k} x^{2k}}$, with $x\in \R$, $\be_{k}=(-1)^{k-1}\al$, $\al>0$ and $k\geq 1$. Then $\f$ satisfies:
\begin{equation}\label{RodNic2}
\vert \f(\om)\vert \;\leq\; C_{k,\al}\me^{ -\varepsilon_{k,\al}|\om|^{\frac{2k}{2k-1}}},
\end{equation}
where $ C_{k,\al} = \frac{|2 k \al|^{\frac{k-1}{2k-1}}}{(2\pi)^{\frac{2k(k-1)}{2k-1}}}$ and $\varepsilon_{k,\al} = (2\pi)^{\frac{2k}{2k-1}}\left(\frac{2k-1}{2k}\right)\left(\frac{1}{2k\al}\right)^{\frac{1}{2k-1}}.$
\end{Cor}
\begin{proof}
Following the path of the proof above, one can easily obtain the analog of \eqref{trenta}, that is
\begin{align*}
|\f(\om)|&{}\leq \left| \frac{(2\pi)^{2k}(-1)^k \om}{2k\be_{k}}\right|^{-\frac{k-1}{2k-1}}\me^{\Re\left\{\nu_j  \int_a^{\om} \left[\frac{(2\pi)^{2k}(-1)^k t}{2k\be_{k}}\right]^{\frac{1}{2k-1}}\ud t \right\}}\\
&{}= \left| -\frac{(2\pi)^{2k} \om}{2k\al}\right|^{-\frac{k-1}{2k-1}}\me^{\Re\left\{\nu_j  \int_a^{\om} \left[-\frac{(2\pi)^{2k}t}{2k\al}\right]^{\frac{1}{2k-1}}\ud t \right\}}\\
&= \frac{(2k\al)^{\frac{k-1}{2k-1}}}{(2\pi)^{\frac{2k(k-1)}{2k-1}}|\om|^{\frac{k-1}{2k-1}}}\me^{ -\delta_{k}\frac{2\pi^{\frac{2k}{2k-1}}}{({2k\al)^{\frac{1}{2k-1}}}}\left( \frac{2k-1}{2k}\right)(\om ^{\frac{2k-1}{2k}}- a^{\frac{2k-1}{2k}})}\\
&\leq C_{a}\frac{(2k\al)^{\frac{k-1}{2k-1}}}{(2\pi)^{\frac{2k(k-1)}{2k-1}}|\om|^{\frac{k-1}{2k-1}}}\me^{ -\delta_{k}\frac{2\pi^{\frac{2k}{2k-1}}}{({2k\al)^{\frac{1}{2k-1}}}}\left(\frac{2k-1}{2k}\right)\om ^{\frac{2k-1}{2k}}}, \qquad \om \in\Omega. 
\end{align*}
In this case $\delta_{k}=1$, since we are in the ``odd'' case described in the proof of \ref{Teorem_1}. Hence the result follows.
\end{proof}
\section{Heat Equation}
\subsection{Heat Equation}
We will focus on the numerical representation of solutions for the following Cauchy problem: 
\begin{equation}\label{Pc_He}
\begin{split}
&\p_t u +(- \rho\Delta) u = 0,\quad (t,x)\in\R\times\R^{d},\\
&u(0,x) = u_0(x),
\end{split}
\end{equation}
where $u_0 \in \s$ and the parameter $\rho$ is the  thermal diffusion. 
The solution $u(t,x)$ to this Cauchy Problem  can be represented by the action of the families of Fourier multipliers $\sigma_{\rho}(t,D)$ on the initial datum as follows:
\begin{equation}\label{eqt18}
u(t,x)= \int_{\R^d}\me^{2\pi i x\cdot\om}\sigma_{\rho}(t,\om)\widehat{u_0}(\om)\ud\om = \F^{-1}\left( \sigma_{\rho}(t,\om)\widehat{u_0}\right)(x),
\end{equation}
where $\sigma_{\rho}(t,x)= \me^{ -4\pi^2 \rho t |x|^2}$.
The super-exponential sparsity of the Gabor matrix for the Heat Equation was proved in \cite{2012arXiv1209.0945C}, and some numerical simulations were shown in \cite[Sec. 6]{2012arXiv1209.0945C}. Here we refine those numerical estimates presenting a new exact representation of the solution via Gabor frames.
\begin{theorem}
For every $t\in\R$ consider the operator $\sigma_{\rho}(t,D)$ defined in \eqref{eqt18}. Let $\mathcal{G}(g,\al,\be)$ be a Gabor frame with a Gaussian window function $g$.
Let $(m,n),(m',n') \in\al\Z^d\times\be\Z^d$ with $\al\be < 1$; then the modulus of the Gabor matrix of the operator $\sigma_{\rho}(t,D)$  is equal to:
\begin{multline}\label{eqt19}
\left|\la \sigma_{\rho}(t,D)\pi(m,n)g, \pi(m',n')g\ra\right| \\= (2+4\pi\rho t)^{-\frac{d}{2}}\me^{  -\pi\left [|n|^2+|n'|^2+\frac{1}{2+4\pi\rho t}\left(|m-m'|^2-|n+n'|^2\right)\right]}
\end{multline}
\end{theorem}
\begin{proof}
The calculation of the Gabor matrix for the Fourier multiplier reduces to 
\begin{equation}\label{eqt20}
 |\la \sigma_{\rho}(t,D)\pi(m,n)g, \pi(m',n')g\ra| =\left| \langle\F^{-1}\left( \me^{-4 \rho t\pi^2|\om|^2}\F({\pi(m,n)g}) \right),\pi(m',n')g\ra\right|
\end{equation}
We recall that for $a>0$
\begin{equation}\label{fouGau}
\widehat{\me^{ -a x^2}}(k)= \sqrt{\frac{\pi}{a}} \me^{\frac{-\pi^2 k^2}{a}}
\end{equation}
Using Plancherel's Theorem and Relation \eqref{fouGau} above, Equation \eqref{eqt20} can be restated as follows
\begin{align*}
{} & \left|\langle \me^{ -4 \rho t\pi^2|x|^2}\F\left({\pi(m,n)g}\right),\F\left({\pi(m',n')g}\right)\ra\right|\\
&=\left| \int_{\R^d}\me^{ -4 \rho t\pi^2|x|^2}T_nM_{-m}g(x)\overline{T_{n'}M_{-m'}g(x) } \ud x\right|\\
&=\left|\int_{\R^d}\me^{-4 \rho t\pi^2|x|^2}\me^{-2\pi i m(x-n)}\me^{-\pi|x-n|^2} \me^{ 2\pi i m'(x-n')}\me^{ -\pi|x-n'|^2}\ud x\right|\\
&= \left|\me^{ -\pi(|n'|^2+|n|^2)}\int_{\R^d}\me^{ -2\pi i x\left[(m-m')+i(n+n')\right]-(2\pi +4 \rho t \pi^2)|x|^2}\ud x\right|\\
&= \left|\me^{ -\pi(|n'|^2+|n|^2)}\F\left(\me^{ -(2\pi +4 \rho t \pi^2)|x|^2}\right)\left[(m-m')+i(n+n')\right]\right|\\
&=\me^{-\pi(|n'|^2+|n|^2)} (2+4\pi \rho t)^{-d/2}\me^{ -\frac{\pi}{2}\cdot\frac{1}{1+2\pi \rho t}\left ((m-m')+i(n+n')\right)^{2}}\\
&=\me^{-\pi(|n'|^2+|n|^2)} (2+4\pi \rho t)^{-d/2}\me^{ -\frac{\pi}{2}\cdot\frac{1}{1+2\pi \rho t}\left(|m-m'|^2-|n+n'|^2\right)}.
\end{align*}
That is the claim.
\end{proof}
\begin{remark}
We claim that our estimate is more accurate than the one in \cite{2012arXiv1209.0945C}. To figured this out, first one can check that  \[\me^{-\frac{\pi}{2}\cdot\frac{1}{1+2\pi \rho t} |n+n'|^{2}}\leq \me^{\pi\cdot\frac{1}{1+2\pi \rho t}(|n'|^2+|n|^2)}\] and thus
\begin{multline*}\me^{-\pi(|n'|^2+|n|^2)}\me^{\frac{\pi}{1+2\pi \rho t}(|n'|^2+|n|^2)}= \me^{-\frac{2\pi^{2}\rho}{1+2\pi \rho t}(|n'|^2+|n|^2)}\leq \me^{-\frac{2\pi^{2}\rho}{1+2\pi \rho t}(|n'-n|^2)}.\end{multline*}
Then, it is easy to see that
\begin{align*}
&{} \me^{-\pi(|n'|^2+|n|^2)} (2+4\pi \rho t)^{-d/2}\me^{ -\frac{\pi}{2}\cdot\frac{1}{1+2\pi \rho t}\left(|m-m'|^2-|n+n'|^2\right)}\\
&\leq (2+4\pi \rho t)^{-d/2}\me^{ -\frac{\pi}{2}\frac{1}{1+2\pi \rho t}\left(|m-m'|^2+|n|^{2}+|n'|^2\right)}\\
&\leq (2+4\pi \rho t)^{-d/2}\me^{ -\varepsilon\left(|m-m'|^2+|n-n'|^{2}\right)},
\end{align*}
Hence \eqref{eqt19} implies
\[|\la \sigma_{\rho}(t,D)\pi(m,n)g, \pi(m',n')g\ra|\mc \me^{ -\varepsilon\left(|m-m'|^2+|n-n'|^2\right)},\]
which is the upper bound for the Gabor Matrix given in \cite{2012arXiv1209.0945C}. This proves that our estimate is more accurate, as expected.\end{remark}
\subsection{Generalized Heat Equation}
The Cauchy problem for the \emph{Generalized Heat Equation} can be stated as follows:
\begin{equation}\label{Pc_Gh}
\begin{split}
&\p_t u +(- \Delta^k) u = 0, \quad (t,x) \,\in\, \R\times\R,\\
&u(0,x) = u_0(x),
\end{split}
\end{equation}
with $u_0\in\mathcal{S}(\R)$.
The solution $u(t,x)$ to this Cauchy Problem  can be represented by the action of the families of Fourier multipliers $\sigma_{k}(t,D)$ on the initial datum as follows:
\begin{equation}\label{FMult}
u(t,x)= \int_{\R^d}\me^{2\pi i x\cdot\om}\sigma_{k}(t,\om)\widehat{u_0}(\om)\ud\om = \F^{-1}\left( \sigma_{k}(t,\om)\widehat{u_0}\right)(x),
\end{equation}
where $\sigma_{k}(t,x)= \me^{t(2\pi i x)^{2k}} = \me^{ t (-1)^{k}(2 \pi x)^{2k}}$.\\
We shall give an estimate from above of the Gabor matrix of $\sigma_{k}(t,D)$.\\
\indent First we need a preliminary result, which gives a bound for the convolution of super-exponential functions.
\begin{lemma}\label{Lemma2}
Let $w_{s,\varepsilon}(z) = \me^{-\varepsilon\vert z\vert ^{\frac{1}{s}}},\; z\in\R^d$ and $s>\frac{1}{2}$. Then \begin{equation}
\big(w_{s,\varepsilon}\;\ast\; w_{s,\varepsilon}\big)(z)\mc \me^{-\varepsilon 2^{-\frac{1}{s}} \vert z\vert ^{\frac{1}{s}}}.
\end{equation}
\end{lemma}
\begin{proof}
Expanding the convolution product between the weight functions one has:
\[\big(w_{s,\varepsilon}\;\ast\; w_{s,\varepsilon}\big)(z) = \int_{\R^d}  \me^{ -\varepsilon\vert x\vert ^{\frac{1}{s}}} \me^{ -\varepsilon \vert z-x\vert ^{\frac{1}{s}}} \ud x. \]
Consider the set $N_{z}:= \big\{x\in\R^d : \vert z-x\vert \leq \frac{\vert z\vert }{2} \big\}$. 
If $x \in N_{z}$ then $\vert x\vert \geq \frac{\vert z\vert }{2}$; therefore 
\[ \me^{-\varepsilon \vert x\vert ^{\frac{1}{s}}} \leq \me^{ -\varepsilon 2^{-\frac{1}{s}} \vert z\vert ^{\frac{1}{s}}}. \]
Now, using the previous results:
\begin{align*}
\big(w_{s,\varepsilon}\ast w_{s,\varepsilon}\big)(z) &=\int_{\R^d}  \me^{-\varepsilon \vert x\vert ^{\frac{1}{s}}} \me^{-\varepsilon \vert z-x\vert ^{\frac{1}{s}}}\ud x\\
&= \int_{N_{z}}  \me^{-\varepsilon \vert x\vert ^{\frac{1}{s}}}\me^{-\varepsilon \vert z-x\vert ^{\frac{1}{s}}} \ud x +\int_{N^{c}_{z}}  \me^{-\varepsilon \vert x\vert ^{\frac{1}{s}}}  \me^{-\varepsilon \vert z-x\vert ^{\frac{1}{s}}}\ud x\\
&\leq  \me^{-\varepsilon 2^{-\frac{1}{s}} \vert z\vert ^{\frac{1}{s}}} \cdot \left( \int_{N_{z}} \me^{-\varepsilon \vert z-x\vert ^{\frac{1}{s}}}\ud x +\int_{N^{c}_{z}}  \me^{-\varepsilon \vert x\vert ^{\frac{1}{s}}}\ud x \right) \\
&\mc \me^{-\varepsilon 2^{-\frac{1}{s}}\vert z\vert ^{\frac{1}{s}}} = w_{s,\varepsilon 2^{-\frac{1}{s}}}.
\end{align*}
as desired.
\end{proof}
\begin{theorem}\label{thm1}
Let $\mathcal{G}(g,\al,\be)$ be a Gabor Frame with a Gaussian window function $g$.
Let $\lambda = (m,n), \nu = (m',n') \in\al\Z\times\be\Z$ with $\al\be < 1$. Then
\[\vert \la \sigma_{k}(t,D) \pi(\lambda)g,\pi(\nu)g\ra\vert  \mc C_{t,k}\me^{ -\tilde{\varepsilon}_{t,k} 2^{-\frac{1}{s}}( \lambda-\nu) ^{\frac{1}{s}}}, \]
with $s = \frac{2k}{2k-1}$, $\tilde{\varepsilon}_{k,t} = \left(\frac{2k-1}{2k}\right)\left(\frac{1}{2k t}\right)^{\frac{1}{2k-1}}2^{-\frac{k}{2k-1}} $ and $C_{t,k} =  |2 k t|^{\frac{k-1}{2k-1}}$ .
\end{theorem} 
\begin{proof}
Considering equality \eqref{FMult} and using the Plancherel's Theorem, it follows that:
\begin{align}
\begin{split}
\vert \langle &\sigma_{k}(t,D)\pi(\lambda)g ,\pi(\nu)g\rangle\vert\\  &= \vert \langle \sigma_{k}(t,\om)\F({\pi(\lambda)g),}\F({\pi(\nu)g)}\rangle\vert\\
&= \left|\int_{\R^d} \me^{ t (-1)^{k}(2 \pi \om)^{2k}}\F\Big(\pi(\lambda)\me^{ -\pi x ^2}\Big)(\om)\overline{\F\Big(\pi(\nu)\me^{ -\pi x ^2}\Big)(\om)}\ud \om\right|\label{eqt25}
\end{split}
\end{align}
Recalling that $\lambda = (m,n)$ and $\nu = (m',n')$, then the Fourier transform of the time-frequency shift reads: 
\[\F\big(\pi(\lambda)g) = T_{n}M_{-m}g = \me^{ -2\pi i m\cdot(\om-n)}g(\om-n).\]
Similar relations works for $\nu$. 
Hence \eqref{eqt25} becomes, 
\begin{align}
\phantom{r}
&= \left\vert \int_{\R^d} \me^{ t (-1)^{k}(2 \pi \om)^{2k}}\me^{ -2\pi i m(\om-n) -\pi (  \om-n ) ^2} \me^{ 2\pi i m'(\om-n') -\pi( \om-n')^2}\ud \om\right\vert \nonumber \\
&= \left\vert \me^{ -\pi(n^2+n'^2)}\int_{\R^d} \me^{ -\al\om^{2k}} \me^{ -2\pi\vert \om\vert ^2} \me^{ -2\pi i \om\big((m-m') + i(n+n')\big)}\ud \om\right\vert\nonumber\\
&= \left\vert \me^{ -\pi(n^2+n'^2)}\F\left(\me^{ -\al\om^{2k}}\me^{ -2\pi\om ^2}\right)\left(\theta\right)\right|,\nonumber
\end{align}
with $\al = (-1)^{k-1}t \,(2\pi)^{2k}$ and $\theta = \left((m-m') + i(n+n')\right).$\\
Now we want to use Corollary \ref{cor2} 
Recall that  \eqref{RodNic2} reads: 
\begin{equation*}
\left\vert \F\left(\me^{ -\al\om^{2k}}\right)\right\vert \;\leq\; C_{k,\al}\me^{ -\varepsilon_{k,\tilde{\al}}|\om|^{\frac{2k}{2k-1}}},
\end{equation*}
where $ C_{k,\al} = \frac{|2 k \al|^{\frac{k-1}{2k-1}}}{2\pi^{k-1}}$, $\varepsilon_{k,\al} = \pi\left(\frac{2k-1}{2k}\right)\left(\frac{1}{2k\al}\right)^{\frac{1}{2k-1}}$ and $|\,\om\,|$ is the modulus of $\om \in \R$. \\
Using $\F(f\cdot g) = (\hat{f}\ast\hat{g})$,  and Lemma \ref{Lemma2}, the last integral can be restated as follows:
\begin{align*}
\label{r}
&= \left\vert \me^{ -\pi(n^2+n'^2)}\left(\F\left(\me^{ t (-1)^{k}(2 \pi x)^{2k}}\right)\ast \F\left(\me^{ -2\pi\om ^2} \right)   \right)\big(\theta\big)\right\vert \\
&\mc  \left\vert C_{k,\al}\me^{ -\pi(n^2+n'^2)}\left(\left(\me^{ -\varepsilon_{k,\al} \vert x\vert ^{\frac{2k}{2k-1}}}\right)\ast\left(\me^{ -\frac{\pi}{2} x ^2} \right)   \right)\big(\theta\big)\right\vert \\
&\mc  \left\vert C_{k,\al}  \me^{ -\pi(n^2+n'^2)}\left(\left(\me^{ -\varepsilon_{k,\al} \vert x\vert ^{\frac{2k}{2k-1}}}\right)\ast\left(\me^{ -\varepsilon_{k,\al} \vert x\vert ^{\frac{2k}{2k-1}}} \right)   \right)\big(\theta\big)\right\vert \\
&= \left\vert C_{k,\al} \me^{ -\pi(n^2+n'^2)}\Big(w_{\frac{2k}{2k-1},\varepsilon_{k,\al}}\ast w_{\frac{2k}{2k-1},\varepsilon_{k,\al}} \Big)\big(\theta\big)\right\vert \\
&\mc  \left\vert C_{k,\al} \me^{ -\pi(n^2+n'^2)}\Big(w_{\frac{2k}{2k-1},\varepsilon_{k,\al} 2^{-\frac{2k}{2k-1}}}\Big)\big((m-m') + i(n+n')\big)\right\vert \\
&\mc  \left\vert C_{k,\al} \me^{ -\pi(n^2+n'^2)}\me^{ -\varepsilon_{k,\al} 2^{-\frac{2k}{2k-1}}\big\vert(m-m') + i(n+n')\big\vert^{\frac{2k}{2k-1}} }\right\vert\\
&= \left\vert C_{k,\al} \me^{ -\pi(n^2+n'^2)}\me^{ -\varepsilon_{k,\al} 2^{-\frac{2k}{2k-1}}\big ((m-m')^2 + (n+n')^2\big)^{\frac{k}{2k-1}} }\right\vert\\
&\mc  \left\vert C_{k,\al} \me^{ -\pi(n^2+n'^2)}\me^{ -\varepsilon_{k,\al} 2^{-\frac{2k}{2k-1}}\big ((m-m')^2)^{\frac{k}{2k-1}} }\right\vert.
\end{align*}
It is well known that $\forall \; a,b\; > 0, p>0$ the following inequality holds:
\[(a+b)^p\leq 2^{p}(a^p+b^p).\]
Thus, putting $a= |m-m'|^2, \; b=|n-n'|^2$ and $p = \frac{k}{2k-1}$ it follows that
\begin{equation}
\Big(\big ((m-m')^2\big)^{\frac{k}{2k-1}}  +\big((n-n')^2\big)^{\frac{k}{2k-1}}\Big)\geq 2^{-\frac{k}{2k-1}}\big((m-m')^2+(n-n')^2\big)^{\frac{k}{2k-1}}.
\end{equation}
This inequality is useful in the calculation before to reach the desired estimate: 
\begin{align*}
&\vert \langle \sigma_{k}(t,D)\pi(m,n)g,\pi(m',n')g\rangle \vert \\ 
&\mc  \left\vert C_{k,\al} \me^{ -\frac{\pi}{2}((n-n')^2)}\me^{ -\varepsilon_{k,\al} 2^{-\frac{2k}{2k-1}}\big ((m-m')^2\big)^{\frac{k}{2k-1}} }\right\vert\\
&\mc  \left\vert C_{k,\al} \me^{ -\varepsilon_{k,\al} 2^{-\frac{2k}{2k-1}}\Big(\big ((m-m')^2\big)^{\frac{k}{2k-1}}  +\big((n-n')^2\big)^{\frac{k}{2k-1}}\Big)}\right\vert\\
&\mc  \left\vert C_{k,\al} \me^{ -\varepsilon_{k,\al} 2^{-\frac{3k}{2k-1}}\big ((m-m')^2+(n-n')^2\big)^{\frac{k}{2k-1}}}\right\vert\\
&\mc  \left\vert C_{k,\al} \me^{ -\varepsilon_{k,\al} 2^{-\frac{3k}{2k-1}}\big ( \lambda-\nu)^{\frac{2k}{2k-1}}}\right\vert,\\
\end{align*} 
where $\lambda = (m,n) \mbox{ and } \nu = (m',n')$.\\
\indent If we exploit the constant $\al = \al(k,t) = t\,(2\pi)^{2k}$ into $C_{k,\al}$ and  $\varepsilon_{k,\al}$, we obtain:
\[ C_{k,\al} = \frac{|2 k \al|^{\frac{k-1}{2k-1}}}{(2\pi)^{\frac{2k(k-1)}{2k-1}}}  = \frac{|2 k t (2\pi)^{2k}|^{\frac{k-1}{2k-1}}}{(2\pi)^{\frac{2k(k-1)}{2k-1}}} =|2 k t |^{\frac{k-1}{2k-1}}.
\]
Analogously 
\begin{align*}
\varepsilon_{k,\al}   2^{-\frac{k}{2k-1}} &= (2\pi)^{\frac{2k}{2k-1}}\left(\frac{2k-1}{2k}\right)\left(\frac{1}{2k (2\pi)^{2k} t}\right)^{\frac{1}{2k-1}}  2^{-\frac{k}{2k-1}}\\
&=\left(\frac{2k-1}{2k}\right)\left(\frac{1}{2 k t}\right)^{\frac{1}{2k-1}}2^{-\frac{k}{2k-1}} .\end{align*}
Finally, renaming the two constants
\[C_{k,t} = |2 k t|^{\frac{k-1}{2k-1}}\qquad \mbox{ and }\qquad \tilde{\varepsilon}_{k,t } = \left(\frac{2k-1}{2k}\right)\left(\frac{1}{2 k t}\right)^{\frac{1}{2k-1}}2^{-\frac{k}{2k-1}},   \]
the result follows.
\end{proof}
\begin{remark}
Thanks to Remark \ref{note3}, it is clear that $\sigma_{k}(x,t) = \me^{-t(2\pi x)^{2k}}$ fulfills the hypothesis of Theorem \ref{G-S decay} with $s=\frac{2k-1}{2k}$. From Theorem \ref{thm1}, we get
\[\vert \la \sigma_{k}(t,D) \pi(\lambda)g,\pi(m',n')g\ra\vert  \mc C_{t,k}\me^{ -\tilde{\varepsilon}_{t,k} 2^{-\frac{1}{s}}\Big((m,n)-(m',n')\Big) ^{\frac{1}{s}}}, \]
with $s = \frac{2k}{2k-1}$ and that is consistent with \eqref{GS_ineq}.\end{remark}
\section{Schr\"odinger equation with hamiltonian $\mathcal{H}_\mathcal{A} = -\frac{1}{4\pi} \Delta u + \pi |x|^2 $}\label{Sec4}
The Cauchy problem for the \emph{Harmonic Repulsor} can be stated as follows:
\begin{equation}\label{Pc_Hr}
\begin{split}
&i\p_t u -\frac{1}{4\pi} \Delta u + \pi |x|^2= 0, \quad (t,x)\,\in\, \R\times\R^{d},\\
&u(0,x) = u_0(x),
\end{split}
\end{equation}
$u_0\in\s$.
The solution can be calculated using the metaplectic operator in Section \ref{section2.4}. 
Here the Hamiltonian $H_{\mathcal{A}}= -\frac{1}{4\pi} \Delta u + \pi |x|^2$ can be written as  $H_{\mathcal{A}} = 2\pi\mathcal{P^{\om}_A}$, with $\mathcal{P^{\om}_A}$
with $B_{j,k} = C_{j,k}  = \delta^i_j$ and $A_{j,k} = D_{j,k}= 0$.
Therefore, the symplectic matrix related to \eqref{Pc_Hr} is
$\mathcal{A}= \left( \begin{array}{rr}
0 & I_d \\ 
I_d & 0
\end{array}\right)$.
We know that the solution of the Cauchy problem is given by  $\mu(\me^{ t \mathcal{A}})u_0(x)$, where $\mu$ is the mataplectic representation. So we have to calculate the exponential of the matrix $\mathcal{A}$. The diagonal decomposition of $\mathcal{A}$ is
\[
\mathcal{A}=\frac{1}{2}\left(\begin{array}{rr}
I_d & -I_d\\
I_d & I_d
\end{array}\right)\left(\begin{array}{rr}
I_d & 0\\
0 & -I_d
\end{array}\right)\left(\begin{array}{rr}
I_d & I_d\\
-I_d & I_d
\end{array}\right)\]
Then we have 
\[
\me^{t\mathcal{A}}=\frac{1}{2}\left(\begin{array}{rr}
I_d & -I_d\\
I_d & I_d
\end{array}\right)\left(\begin{array}{cc}
\me^{tI_d} & 0\\
0 & \me^{-tI_d}
\end{array}\right)\left(\begin{array}{rr}
I_d & I_d\\
-I_d & I_d
\end{array}\right). \]
It is easy to see that $\me^{\pm tI}= \me^{\pm t}\left(\begin{array}{rr}
I_d & 0\\
0 & I_d
\end{array}\right)$. Thus,
\[\me^{t\mathcal{A}}= \left(\begin{array}{rr}
\frac{\me^t+\me^{-t}}{2}I_d &\frac{\me^t-\me^{-t}}{2}I_d\\
\frac{\me^t-\me^{-t}}{2}I_d & \frac{\me^t+\me^{-t}}{2}I_d
\end{array}\right).\]
Using the definition of the hyperbolic sine and cosine, we obtain
 \[\me^{t\mathcal{A}}= \left( \begin{array}{rr}
\cosh(t)I_d & \sinh(t)I_d \\ 
\sinh(t)I_d & \cosh(t)I_d
\end{array}\right).\]
Now, we can use \eqref{Metaplectic_Repr_eqt} to calculate the solution of
\eqref{Pc_Hr} with initial datum $u_0 = \pi(m,n) g$, that is 
\begin{align}
\begin{split}
\label{t9}
& u(t,x) = \mu(\me^{t\mathcal{A}})\pi(m,n)g \\&= \frac{1}{\cosh(t)}\int_{\R^d} \me^{ \pi i \tanh(t)x^2 + \frac{2\pi i x \om}{\cosh(t)}-\pi i \tanh(t)\om^2} \F\left(\pi(m,n)g\right)(\om)\ud \om.
\end{split}
\end{align}

Fix $t\in\R$, in order to calculate the Gabor matrix of the operator $T$, we compute first \eqref{t9}.
\begin{lemma}\label{Lemma41}
Consider the metaplectic operator \eqref{t9} and the time-frequency shifts of the Guassian  $\pi(m,n)g(x) = M_n T_m \me^{-\pi |x|^2}$, where $(m,n)\in \Lambda$ with $\Lambda =  \al\Z^d\times\be\Z^d$ and $\al\be<1$. Then one has
\begin{align*}
T \pi(m,n)g(x) &=  C_t \me^{-\frac{\pi (m+in)^2 \cosh(t)}{\cosh(t)+i\sinh(t)}}  \me^{-\pi x^2 \left[ \frac{\cosh(t)-i\sinh(t)}{\cosh(t)+i\sinh(t)}\right]+\frac{2\pi x(m+ i n)}{\cosh(t)+i\sinh(t)}}.
\end{align*}
where $C_t = \left(\frac{1}{\cosh(t)+i\sinh(t)}\right)^{\frac{d}{2}} \me^{ -\pi n^2 + 2\pi i m\cdot\, n}.$
\end{lemma}
\begin{proof}
We expand \eqref{t9}:
\begin{align}
 u(t,x)  ={}& \left(\frac{1}{\cosh(t)}\right)^{\frac{d}{2}}\int_{\R^d} \me^{ \pi i \tanh(t)x^2 + \frac{2\pi i x\om}{\cosh(t)}-\pi i \tanh(t) \om^2}\F\left(\pi(m,n)g\right)(\om)\ud \om\nonumber \\
  ={}& \left(\frac{1}{\cosh(t)}\right)^{\frac{d}{2}}\me^{ \pi i \tanh(t) x^2}\int_{\R^d}\me^{ 2\pi i \left[\frac{x\cdot \om}{\cosh(t)}-\frac{\tanh(t)\om^2}{2}\right]}T_n M_{-m}\hat{g}(\om)\ud\om\nonumber\\
  ={}& \left(\frac{1}{\cosh(t)}\right)^{\frac{d}{2}}\me^{ \pi i \tanh(t) x^2}\nonumber\\
 \cdot{}&\int_{\R^d}\me^{ 2\pi i \left[\frac{x \om}{\cosh(t)}-\frac{\tanh(t)\om^2}{2}\right]} \me^{-2\pi i m (\om-n)-\pi|\om-n|^2}\ud\om\nonumber\\
 ={}&  \left(\frac{1}{\cosh(t)}\right)^{\frac{d}{2}}\me^{ \pi i \tanh(t) x^2}\me^{ -\pi n^2 + 2\pi i m\, n}\nonumber\\
 &\cdot \int_{\R^d}\me^{ -2\pi i \om\left[ m+ i n  - \frac{x}{\cosh(t)}\right]}\me^{ -\pi\om^2\left(1+i \tanh(t)\right)}\ud\om.\label{t10}
\end{align}
Using \eqref{fouGau},  we can restate \eqref{t10} as follows
\begin{equation}\label{eq11}
  C_t\, \me^{ \pi i \tanh(t) x^2}\me^{ -\pi {\left[ m+ i n  - \frac{x}{\cosh(t)}\right]^2 \left(\frac{\cosh(t)}{\cosh(t) + i\sinh(t)}\right) }},
\end{equation}
with $C_t = \left(\frac{1}{\cosh(t)+i\sinh(t)}\right)^{\frac{d}{2}}\me^{ -\pi n^2 + 2\pi i m\cdot\, n}$.\\
Since
\begin{align*}
u(t,x) = {} &  C_t\, \me^{ \pi i \tanh(t) x^2-\pi {\left[ m+ i n  - \frac{x}{\cosh(t)}\right]^2 \left(\frac{\cosh(t)}{\cosh(t) + i\sinh(t)}\right) }}\\
 = {} & C_t\, \me^{ \pi i \frac{\sinh(t)}{\cosh(t)} x^2} \me^{ -\pi \left[ (m+ i n)^2  \right]\left(\frac{\cosh(t)}{\cosh(t) + i\sinh(t)}\right)}\\
\cdot& \me^{ -\pi \left[-2(m+ i n)\frac{x}{\cosh(t)}+\frac{x^2}{\cosh(t)^2}\right]\left(\frac{\cosh(t)}{\cosh(t) + i\sinh(t)}\right)}\\ 
 = {} & C_t\, \me^{ -\pi x^2 \left[ \frac{1-i\sinh(t)\left(\cosh(t)+ i\sinh(t)\right)}{\cosh(t)\left(\cosh(t)+i\sinh(t)\right)}\right]}\\
 &\cdot  \me^{ \frac{2\pi x(m+ i n)}{\cosh(t)+i\sinh(t)}}\me^{ -\frac{\pi (m+ i n)^2 \cosh(t)}{\cosh(t)+i\sinh(t)}}\\
 = {}& C_t\, \me^{ -\pi x^2 \left[ \frac{\cosh(t)^2-i\sinh(t)\cosh(t)}{\cosh(t)(\cosh(t)+i\sinh(t))}\right]} \me^{ \frac{2\pi x(m+ i n)}{\cosh(t)+i\sinh(t)}}\\ &\cdot \me^{ -\frac{\pi (m^2+ 2 i m n-n^2) \cosh(t)}{\cosh(t)+i\sinh(t)}}\\
={} & C_t \me^{ -\pi x^2 \left[ \frac{\cosh(t)^2-i\sinh(t)\cosh(t)}{\cosh(t)(\cosh(t)+i\sinh(t))}\right]} \me^{ \frac{2\pi x(m+ i n)}{\cosh(t)+i\sinh(t)}}\\&\cdot  \me^{ -\frac{\pi (m^2+ 2 i m n-n^2) \cosh(t)}{\cosh(t)+i\sinh(t)}}\\
={} & \tilde{C}_t \me^{ -\pi x^2 \left[ \frac{\cosh(t)-i\sinh(t)}{\cosh(t)+i\sinh(t)}\right]} \me^{ \frac{2\pi x(m+ i n)}{\cosh(t)+i\sinh(t)}},\\
\end{align*}
where \[
\tilde{C_t}=\left(\frac{1}{\cosh(t)+i\sinh(t)}\right)^{\frac{d}{2}} \me^{ -\pi n^2 + 2\pi i m n - \frac{\pi (m^2+ 2 i m n-n^2) \cosh(t)}{\cosh(t)+i\sinh(t)}}.\]
Hence the result is proved.
\end{proof}

\indent The computation of the Gabor matrix $T_{m,n,m',n'}=\la T \pi(m,n)g,\pi(m',n')g\ra$ of the metaplectic operator $T_t$, reduces now to compute the inner product above.
\begin{theorem}
Let $T_t$  be the operator defined in \eqref{t9} and $(m,n), (m',n')\in \Lambda$ with $\Lambda = \al\Z^d\times\be\Z^d$,  $\al\be<1$. Then 
\begin{align}
\label{Sol_Hr}
\begin{split}
T_{m,n,m',n'} = C_t \me^{ -\frac{\pi}{2}\left[|m|^2+|n|^2+|n'|^2+|m'|^2 + 2\tanh(mn-m'n') -2(mm'-nn')\right]},
\end{split}
\end{align}
where $C_t =  \frac{\me^{i\psi}}{\left(2\cosh(t)\right)^{\frac{d}{2}}} $ and $\vert \me^{i\psi}\vert = 1$.
\end{theorem}
	
\begin{proof}
Using Lemma \ref{Lemma41}:
\begin{align*}
& T_{m,n,m',n'} = \\
 ={}& \la \tilde{C}_t \me^{ -\pi x^2 \left[ \frac{\cosh(t)-i\sinh(t)}{\cosh(t)+i\sinh(t)}\right]+\frac{2\pi x(m+ i n)}{\cosh(t)+i\sinh(t)}},\pi(m',n')g\ra\\
 ={}& \tilde{C}_t\int_{\R^d}\me^{-\pi x^2 \left[ \frac{\cosh(t)-i\sinh(t)}{\cosh(t)+i\sinh(t)}\right]+\frac{2\pi x(m+ i n)}{\cosh(t)+i\sinh(t)}}\overline{M_{n'}T_{m'}\me^{-\pi |x|^2}}\ud x\\
 ={}& \tilde{C}_t\int_{\R^d} \me^{ -\pi x^2 \left[ \frac{\cosh(t)-i\sinh(t)}{\cosh(t)+i\sinh(t)}\right]}\\ 
 &\cdot \me^{\frac{2\pi x(m+ i n)}{\cosh(t)+i\sinh(t)}-2\pi i n'x -\pi |x-m'|^2}\ud x\\
 ={}& \tilde{C}_t \me^{ -\pi {m'}^2}\int_{\R^d} \me^{ -\pi x^2 \left[ \frac{\cosh(t)-i\sinh(t)}{\cosh(t)+i\sinh(t)}\right]}\\ 
 &\cdot \me^{\frac{2\pi x(m+ i n)}{\cosh(t)+i\sinh(t)}-2\pi i n'x-\pi x^2 +2\pi m'x}\ud x\\
 ={}& \tilde{C}_t \me^{ -\pi {m'}^2}\int_{\R^d} \me^{-\pi x^2 \left[ \frac{\cosh(t)-i\sinh(t)}{\cosh(t)+i\sinh(t)}+1\right]} \\&\cdot \me^{\frac{2\pi x(m+ i n)}{\cosh(t)+i\sinh(t)} -2\pi i x(im'+n')}\ud x\\
 ={}& \tilde{C}_t \me^{ -\pi {m'}^2}\int_{\R^d} \me^{ -\pi x^2 \left[ \frac{2\cosh(t)}{\cosh(t)+i\sinh(t)}\right]} \\&\cdot \me^{ -2\pi i x\frac{\left[(i m- n)+(im'+n')\left(\cosh(t)+i\sinh(t)\right)\right]}{\cosh(t)+i\sinh(t)}}\ud x.\\
\end{align*}
Using \eqref{fouGau}, we obtain:
\begin{multline}
\label{t13}
T_{m,n,m',n'} =  \tilde{C}_t \me^{ -\pi {m'}^2}\left(\frac{\cosh(t)+i\sinh(t)}{2\cosh(t)}\right)^{\frac{d}{2}}\\ \cdot \me^{ -\frac{\pi}{2}\cdot\frac{\left[(im-n)+(im'+n')(\cosh(t)+i\sinh(t))\right]^2}{\cosh(t)\left(\cosh(t)+i\sinh(t)\right)}}.
\end{multline} 
Expanding $\tilde{C}_t$, \eqref{t13} becomes:
\begin{align}
\begin{split}
T_{m,n,m',n'}= \frac{1}{(2\cosh(t))^{\frac{d}{2}}}\me^{ -\frac{\pi}{2}\,\frac{1}{\cosh(t) (\cosh^2(t)+\sinh^2(t))}\Phi(m,n,m',n',t)}.\label{t15}
\end{split}
\end{align}
Now we have to give a clear formulation of $\Phi$. 
The calculations that follow do not take care of the imaginary part which is always contained in a  real-valued function $\psi$.
\begin{align*}
\Phi = {}& 2\left[(|m'|^2+|n|^2)+2 i mn\right]\cdot\left[\cosh(t)	\cdot \left(\cosh^2(t)+\sinh^2(t)\right)\right] \\
&+ \left\{2(m+in)^2\cosh^2(t)\right\}\left(\cosh(t) - i\sinh(t)\right),\\
&+ \left\{\left[(im-n)+(im'+n')(\cosh(t)+i\sinh(t))\right]^2\right\}\\ 
&\cdot \left(\cosh(t) - i\sinh(t)\right),\\
= {} & 2(|m'|^2+|n|^2)\left(\cosh^3(t)+\cosh(t)\sinh^2(t)\right) \\
&+ 2(|m|^2 +2i m n -|n|^2)\cdot(\cosh^3(t) - i\cosh^2(t)\sinh(t))\\
&+ (im-n)^2\cdot\left(\cosh(t) - i\sinh(t)\right) -2(im-n)(im'+n')\\
&\cdot (\cosh^2(t)+\sinh^2(t))\\
&+ (im'+n')^2(\cosh(t)+i\sinh(t))\,\left(\cosh^2(t) +\sinh^2(t)\right)+i\psi\\
= {}& 2(|m'|^2+|n|^2)\left(\cosh^3(t)+\cosh(t)\sinh^2(t)\right) + 2(|m|^2-|n|^2)\cosh^3(t)\\ &+4mn\cosh^2(t)\sinh(t)+ (-m^2+n^2)\cosh(t) -2mn\sinh(t) \\&-2\left( -mm'-nn'-imn'-im'n)\,\left(\cosh^2(t) +\sinh^2(t)\right)\right)\\
&+ (-|m'|^2+|n'|^2)\cosh(t)\left(\cosh^2(t) +\sinh^2(t)\right) \\ &-2\,m'n'\sinh(t)\,\left(\cosh^2(t) +\sinh^2(t)\right)+i\psi\vspace{10cm}\\
= {} & \left[(|m|^2+|n|^2)\cosh(t) + 2mn\sinh(t)+2(mm'+nn')\right]\\ &\cdot\left(\cosh^2(t) +\sinh^2(t)\right)\\
&+ \left[(|m'|^2+|n'|^2)\cosh(t)- 2m'n'\sinh(t)\right]\cdot\left(\cosh^2(t) +\sinh^2(t)\right)+i\psi. \\
\end{align*}
Finally, using \eqref{t15} the Gabor matrix can be expressed as
\[T_{m,n,m',,n'} = C_t \me^{ -\frac{\pi}{2}\left[|m|^2+|n|^2+|n'|^2+|m'|^2 +2\tanh(t)(mn-m'n')-2(mm'-nn')\right]},\]
where $C_t =  \frac{\me^{ i\psi}}{(2\cosh(t))^{\frac{d}{2}}} $ and $|\me^{ i\psi}|=1$.
Thus, the Theorem is proved.
\end{proof} 
\section{Numerical Result}
In this section we show numerical examples to test the fastness of the Gabor coefficients' decay, in dimension $d=1,2,3.$ In dimension $d=1$ we will show the magnitude of the Short-Time Fourier transform of the solutions, i.e. the spectrogram of the solutions to the Cauchy problems of the previous sections. The initial datum we will use is provided by $M_nT_m g$ with $g$ Gaussian function. We shall represent the behavior of the solution in the phase space at different instants of time.  
In all our examples we will use a lattice on $\Z^{2d}$ with parameters $\al = 1, \be = \frac{1}{2}$. In this way the Gabor system $\mathcal{G}(\me^{-\pi|x|^2},1,\frac{1}{2})$ is a frame for $\ld$.
\subsection{Heat Equation}
Numerical evaluations of the Gabor matrix for this problem are already treated in \cite{2012arXiv1209.0945C}. Here, using the exact representation of the Gabor matrix, namely 
\begin{multline*}
\left|\la \sigma_{\rho}(t,D)\pi(m,n)g , \pi(m',n')g\ra\right| = \\
(2+4\pi\rho t )^{-\frac{d}{2}} \me^{  -\pi\left [|n|^2+|n'|^2+\frac{1}{2+4\pi\rho t}\left(|m-m'|^2+|n+n'|^2\right)\right]}.
\end{multline*}
we obtain a faster decay. Moreover, the equation above clarify that the growth of the diffusion factor $\rho$ and the time $t$ cause the same diffusive effect on the solution.
In fact, if we fix the time variable and we let $\rho$ increase, we see a diffusion in the space variable, as shown in Figure \ref{Diff}.
Notice that it is equivalent to fix $\rho= 1$ and let the time $t$ grow, as shown by \ref{Figure_He}.
\begin{figure}\label{Diff}
\vspace{1cm}
\begin{center}{\bf Dependence of the coefficients decay from the Thermal Diffusion.}
\end{center}
\begin{minipage}[h]{0.47\textwidth}
\vspace{1cm}
\includegraphics[width=1\textwidth]{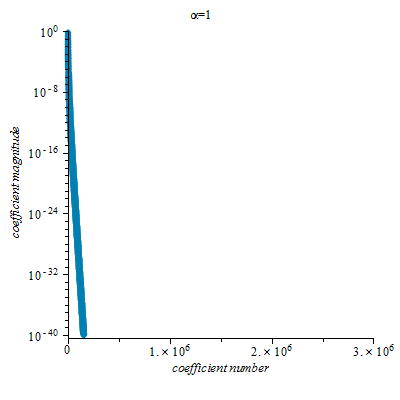}
\end{minipage}\hfill
\begin{minipage}[h]{0.47\textwidth}
\vspace{1cm}\includegraphics[width=1\textwidth]{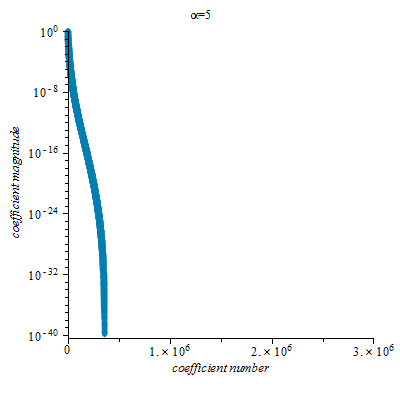}
\end{minipage}
\caption{
The dissipative effect caused by the thermal diffusion $\alpha$.}
\end{figure}
\begin{figure}\label{Figure_He}
\vspace{1cm}
\begin{center}{\bf Magnitude of the coefficient for the Heat Equation at different instants of time $t$, in dimension $d=2$.}
\end{center}
\begin{minipage}[h]{0.47\textwidth}
\vspace{1cm}
\includegraphics[width=1\textwidth]{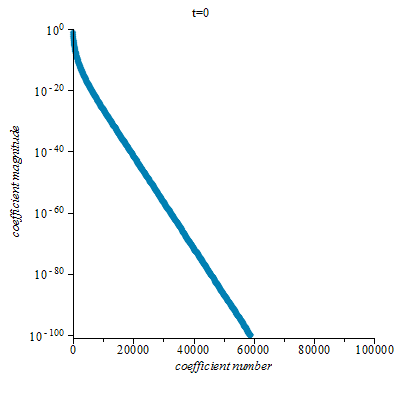}
\end{minipage}\hfill
\begin{minipage}[h]{0.47\textwidth}
\vspace{1cm}\includegraphics[width=1\textwidth]{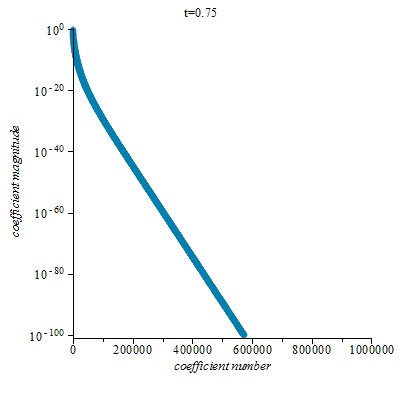}
\end{minipage}\hfill
\begin{minipage}[h]{0.47\textwidth}
\vspace{2cm}\includegraphics[width=1\textwidth]{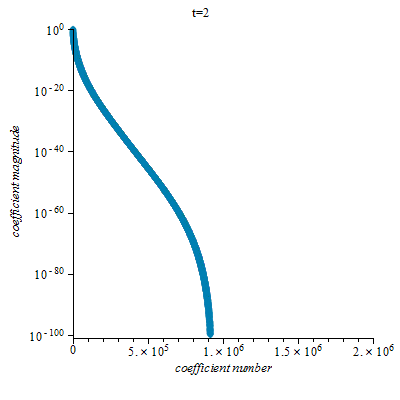}
\vspace{0.5cm}
\end{minipage}\hfill
\begin{minipage}[h]{0.47\textwidth}
\vspace{2cm}\includegraphics[width=1\textwidth]{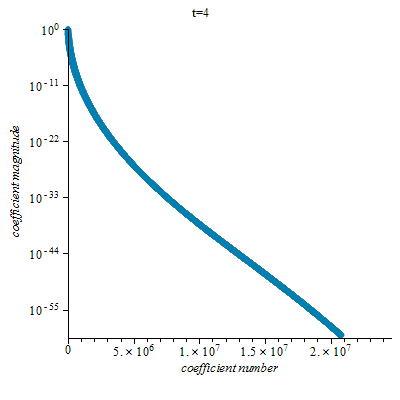}
\vspace{0.5cm}
\end{minipage}\hfill
\caption{
The dissipative effect grows together with the time $t$. Similar effects are observable by increasing the thermal diffusion $\rho$.}
\end{figure}
\begin{figure}
\begin{center}{\bf Magnitude of the coefficient for the Heat Equation at different instants of time $t$, in dimension $d=3$.}
\end{center}
\begin{minipage}[h]{0.47\textwidth}
\vspace{1cm}
\includegraphics[width=1\textwidth]{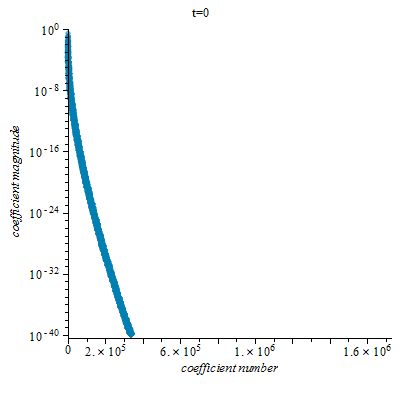}
\end{minipage}\hfill
\begin{minipage}[h]{0.47\textwidth}
\vspace{1cm}\includegraphics[width=1\textwidth]{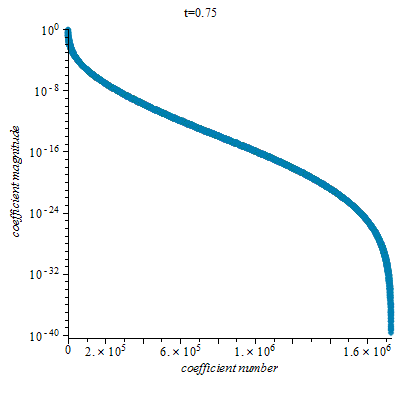}
\end{minipage}
\caption{}
\end{figure}
\subsection{Generalized Heat Equation}
In dimension $d=1$, the Gabor matrix associated to the Cauchy problem \eqref{Pc_Gh} fulfills the estimate:
\[\vert \la \sigma_{k}(t,D) \pi(m,n)g,\pi(m',n')g\ra\vert  \mc C_{t,k}\me^{ -\tilde{\varepsilon}_{t,k} 2^{-\frac{1}{s}}(\lambda-\nu)^{\frac{1}{s}}}, \]
with $s = \frac{2k}{2k-1}$, $\tilde{\varepsilon}_{k,t} = \left(\frac{2k-1}{4k}\right)\left(\frac{1}{4\pi k t}\right)^{\frac{1}{2k-1}}2^{-\frac{k}{2k-1}} $ and $C_{t,k} =  |4\pi k t|^{\frac{k-1}{2k-1}}$ .
Using this equation we show the coefficients' decay in dimension $d=1$ together with their dependence from $k$ and $t$.
\begin{figure}[h!]\label{Figure_Gen_He}
\vspace{1cm}
\begin{center}{\bf Magnitude of the coefficient for the Generalized Heat Equation at different instants of time $t$ and Laplacian powers $k$, in dimension $d=1$.}
\end{center}
\begin{minipage}[h]{0.47\textwidth}
\vspace{1cm}
\includegraphics[width=1\textwidth]{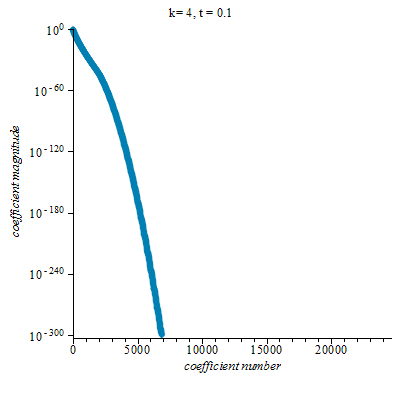}
\end{minipage}\hfill
\begin{minipage}[h]{0.47\textwidth}
\vspace{1cm}\includegraphics[width=1\textwidth]{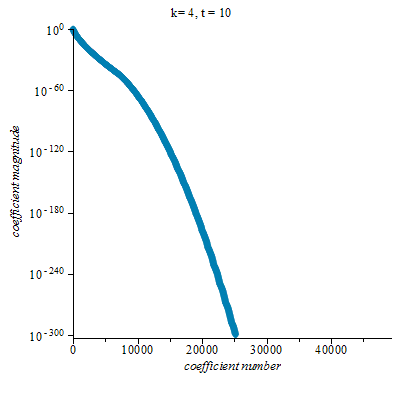}
\end{minipage}\hfill
\begin{minipage}[h]{0.47\textwidth}
\vspace{2cm}\includegraphics[width=1\textwidth]{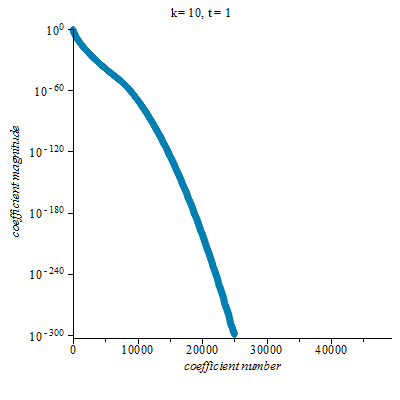}
\vspace{0.5cm}
\end{minipage}\hfill
\begin{minipage}[h]{0.47\textwidth}
\vspace{2cm}\includegraphics[width=1\textwidth]{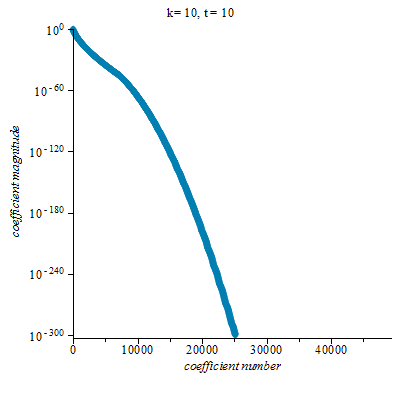}
\vspace{0.5cm}
\end{minipage}\hfill
\caption{This picture shows that the decay obtained at different instant times and powers $k$ is almost identical.}
\end{figure}
\subsection{Harmonic Repulsor}
The exact expression of the Gabor matrix related to the Harmonic Repulsor is given by equation \eqref{Sol_Hr}, that is
 \[\vert T_{m,n,m',n'}\vert = C\me^{ -\frac{\pi}{2}\left[|m|^2+|n|^2+|n'|^2+|m'|^2 +2\tanh(t)(mn-m'n')-2(mm'-nn')\right]},\]whit $C = \frac{1}{(2\cosh(t))^{\frac{d}{2}}}$.
Figure \ref{Figure_REPU} shows that using this expression we obtain huge decays of the coefficients. In dimension $d=2$ and $d=3$ we obtain results similar to those in \cite{MR2502369} for the case of the harmonic oscillator.
\begin{figure}[h!]\label{Figure_REPU}
\vspace{1cm}
\begin{center}{\bf Magnitude of the coefficient for the Harmonic Repulsor at different instants of time $t$, in dimension $d=2$.}
\end{center}
\begin{minipage}[h]{0.47\textwidth}
\vspace{1cm}
\includegraphics[width=1\textwidth]{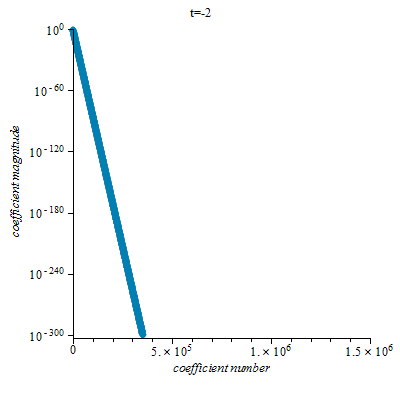}
\end{minipage}\hfill
\begin{minipage}[h]{0.47\textwidth}
\vspace{1cm}\includegraphics[width=1\textwidth]{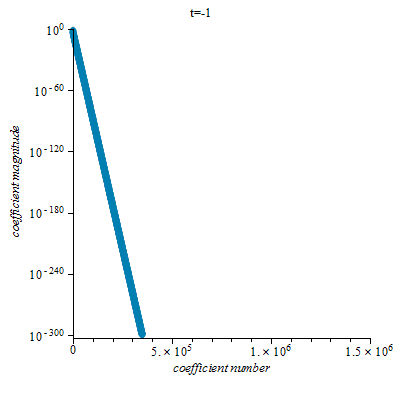}
\end{minipage}\hfill
\begin{minipage}[h]{0.47\textwidth}
\vspace{2cm}\includegraphics[width=1\textwidth]{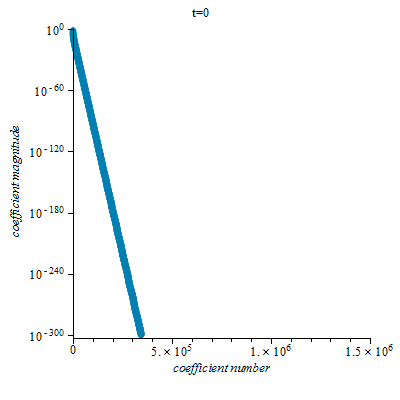}
\vspace{0.5cm}
\end{minipage}\hfill
\begin{minipage}[h]{0.47\textwidth}
\vspace{2cm}\includegraphics[width=1\textwidth]{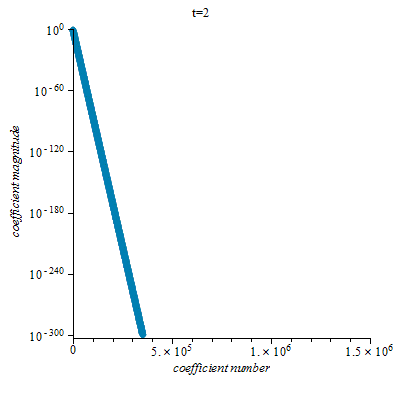}
\vspace{0.5cm}
\end{minipage}\hfill
\caption{This picture shows that the decay obtained at different instants of time is almost identical.}
\end{figure}
\begin{figure}[h!]\label{Figure_REPU_3D}
\vspace{1cm}
\begin{center}{\bf Magnitude of the coefficient for the Harmonic Repulsor at different instants of time $t$, in dimension $d=3$.}
\end{center}
\begin{minipage}[h]{0.47\textwidth}
\vspace{1cm}
\includegraphics[width=1\textwidth]{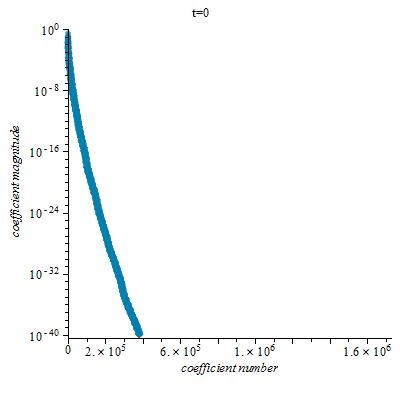}
\end{minipage}\hfill
\begin{minipage}[h]{0.47\textwidth}
\vspace{1cm}\includegraphics[width=1\textwidth]{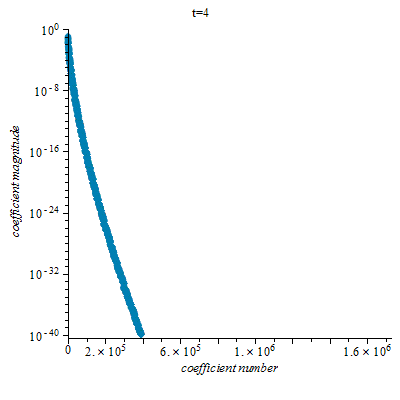}
\end{minipage}
\caption{This picture shows that the decay obtained at different instants of time is almost identical.}
\end{figure}
Figure \ref{Figure_REP_CUNT}  resemble the behavior of the Harmonic Repulsor in the phase space.
The Gaussian approaches the origin from the south-east and then goes to north-est. As the picture shows, although the Gaussian spread in the spatial variable, it remains concentrated in the Time-Frequency domain.
\begin{figure}[htb!]\label{Figure_REP_CUNT}
\vspace{1cm}
\begin{center}{\bf Contourplots of the STFT of Harmonic Repulsor's solution at different instants of time $t$.}
\end{center}
\begin{minipage}[h]{0.47\textwidth}
\vspace{1cm}
\includegraphics[width=1\textwidth]{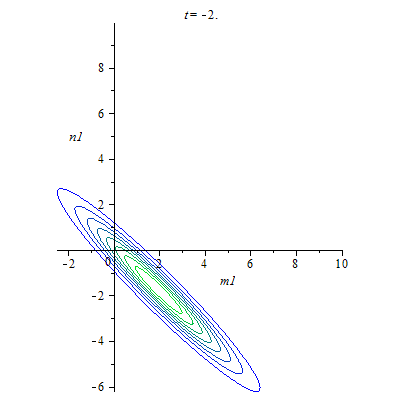}
\end{minipage}\hfill
\begin{minipage}[h]{0.47\textwidth}
\vspace{1cm}\includegraphics[width=1\textwidth]{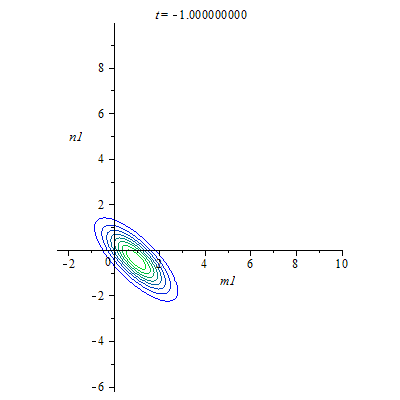}
\end{minipage}\hfill
\begin{minipage}[h]{0.47\textwidth}
\vspace{2cm}\includegraphics[width=1\textwidth]{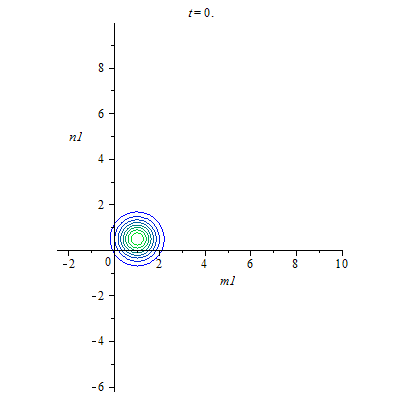}
\vspace{0.5cm}
\end{minipage}\hfill
\begin{minipage}[h]{0.47\textwidth}
\vspace{2cm}\includegraphics[width=1\textwidth]{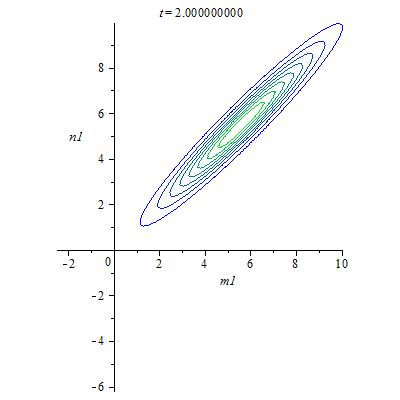}
\vspace{0.5cm}
\end{minipage}\hfill
\caption{
Contour plot for the Short-Time Fourier Transform of the solution of \eqref{Pc_Hr} in
dimension d = 1 at different instants of time, with initial datum and
window $u(x) = g(x) = \me^{-\pi|x|^2}$.}
\end{figure}
\clearpage
\section*{Acknowledgments}
 I am sincerely grateful to Professor E. Cordero for the fruitful  discussion, valuable advices, constructive
criticism and constant review of this work. I would like to thank Professors E.Cordero and L. Rodino for inspiring this paper. I also wish to thank M. Borsero for the useful suggestions, and the final review aimed at improving the readability of the paper. Finally, I am thankful for the enormous work of the anonymous reviewer who suggested important corrections that gave consistency to the paper.
\addcontentsline{toc}{chapter}{References}
\bibliography{bib_mag}
\nocite{*}
\bibliographystyle{plain}
\end{document}